\documentclass[a4paper,11pt]{amsart}
\usepackage[utf8]{inputenc}
\usepackage[T1]{fontenc}
\usepackage{lmodern}
\usepackage{amsmath, amsthm, amsfonts, amssymb} 
\usepackage{graphicx}
\usepackage{float}
\usepackage{hyperref}
\usepackage[margin=2.9cm]{geometry}
\usepackage{biblatex}
\usepackage{enumerate}
\usepackage{caption}

\newcommand{\R}{\mathbb R}
\newcommand{\N}{\mathbb N}
\newcommand{\Z}{\mathbb Z}
\newcommand{\C}{\mathbb C}

\newcommand{\set}[1]{\left\{#1\right\}}
\newcommand{\abs}[1]{\left|#1\right|}
\newcommand{\norm}[1]{\left\lVert#1\right\rVert}

\DeclareMathOperator{\id}{id}
\DeclareMathOperator{\dist}{dist}
\DeclareMathOperator{\supp}{supp}
\DeclareMathOperator{\re}{Re}

\newtheorem{theorem}{Theorem}
\newtheorem*{main}{Main Theorem}
\newtheorem{definition}{Definition}
\newtheorem{proposition}{Proposition}
\newtheorem{lemma}{Lemma}
\newtheorem*{corollary}{Corollary}

\begin{document}

\title{Carleman approximation by non-critical functions on Riemann surfaces}
\author[B. U\v{c}akar]{Beno U\v{c}akar $^{\dag}$}
\address{B. U\v{c}akar: Institute for Mathematics, Physics and Mechanics and Faculty of Mathematics and Physics, University of Ljubljana, Jadranska 19, 1000 Ljubljana, Slovenia}
\email{beno.ucakar@imfm.si}
\thanks{$^{\dag}$  Supported by the research program P1-0291 from ARIS, Republic of Slovenia}
\date{}

\begin{abstract}
    We present the class of semi-admissible subsets of an open Riemann surface on which Carleman approximation by non-critical holomorphic functions is possible.
    In particular we characterize closed sets with empty interior on which continuous functions can be approximated by non-critical holomorphic ones. 
    We also consider a different approach, which in some cases gives uniform approximation by non-critical holomorphic functions on more general sets than semi-admissible ones.
\end{abstract}

\maketitle

\section{Introduction} \label{introduction}

Carleman's approximation theorem states that for any complex-valued continuous function $f \colon \R \to \C$ 
and any continuous positive-valued function $\varepsilon \colon \R \to (0,\infty)$, there exists an entire function $F \colon \C \to \C$ such that $\abs{F(x)-f(x)} < \varepsilon(x)$ holds for every $x \in \R$.
This result was proven by T.~Carleman in \cite{OGCarleman}. 
The distinguishing feature of Carleman approximation is that $\varepsilon$ is a function, 
thus the difference between the original function $f$ and the approximating function $F$ can be controlled pointwise. 
In particular, the difference can be made arbitrarily small as we go off to infinity along the real line.
Thus this is a stronger result than just uniform approximation, where $\varepsilon$ would be a constant function.
Subsets of the complex plane which admit Carleman approximation, for the precise definition see Section \ref{ApproximationOnRS}, were first characterized by A.~A.~Nersesyan in \cite{Nersesyan1} and \cite{Nersesyan2}.
Carleman approximation was later generalized to open Riemann surfaces. Sets which admit Carleman approximation in this setting were characterized by A.~Boivin in \cite{Boivin1986}, see also Section \ref{ApproximationOnRS}.
For a general overview of the theory of holomorphic approximation we refer the reader to the survey article \cite{LegacyOfWeirstrass}.
 
Motivated by the work in \cite{franc} we consider the problem of Carleman approximation by non-critical functions on open Riemann surfaces, namely under what assumptions can one require that the global approximating function has no critical points.
In the case where the open Riemann surface is just the complex plane, this theory might have applications in complex dynamics, where it would serve as a new tool to construct entire functions without critical points with interesting dynamical properties.

This problem is not completely open. 
Using results about approximation of directed minimal immersion by I.~Castro-Infantes and B.~Chenoweth in \cite{BrettCastro}, one can achieve non-critical Carleman approximation of functions with sufficient regularity assumptions on so called \emph{admissible} sets, see Section \ref{ApproximationOnRS}. 
Here we provide a different approach which gives a positive result for functions of the class $\widetilde{\mathcal{A}}$, see Definition \ref{Atilde}, on a more general class of sets which we call \emph{semi-admissible}, see Definition \ref{semi-admissible}.
The proof is in essence similar to the original proof by Carleman, but modified to ensure that the approximating function is non-critical. 
The main result of this paper is the following.

\begin{main}
    Let $X$ be an open Riemann surface and $X^*$ the one point compactification of $X$.
    Let $E \subseteq X$ be a semi-admissible set such that $X^* \setminus E$ is connected and locally connected. 
    Let $f \in \widetilde{\mathcal{A}}(E)$ be non-critical and let $\varepsilon \colon E \to (0, \infty)$ be a continuous positive-valued function.
    Then there exists a global non-critical holomorphic function $F \in \mathcal{O}(X)$ such that we have $\abs{F(p) - f(p)} < \varepsilon(p)$ for each $p \in E$. 
\end{main}

As an immediate consequence of the main theorem, we get the following corollary.

\begin{corollary}
    Let $E \subseteq X$ be a closed set with empty interior, such that $X^*\setminus E$ is connected and locally connected.
    Let $f \in \mathcal{C}(E)$ be a complex-valued continuous function on $E$ and let $\varepsilon \colon E \to (0, \infty)$ be a continuous positive-valued function.
    Then there exists a global non-critical holomorphic function $F \in \mathcal{O}(X)$ such that we have $\abs{F(p) - f(p)} < \varepsilon(p)$ for each $p \in E$. 
\end{corollary}

Since $X^*\setminus E$ being connected and locally connected is a necessary condition even for uniform approximation, 
this characterizes all the closed sets with empty interior on which non-critical approximation is possible.

We provide a short outline of the paper.
In Section \ref{notation} we will introduce some notation. 
Section \ref{ApproximationOnRS} is a short review of some notions from approximation theory on Riemann surfaces.
In Section \ref{terminology} we will introduce the notion of a semi-admissible set and the class of functions $\widetilde{\mathcal{A}}(E)$ on a semi-admissible set $E$.
Section \ref{MainProof} is dedicated to the proof of the main theorem.
In Section \ref{Uniform} we discuss the problem of non-critical approximation on general sets of Carleman approximation and present a different
approach which yields non-critical uniform approximation on certain sets of Carleman approximation which are not semi-admissible.
 
\subsection*{Acknowledgements}
I want to thank Luka Boc Thaler, Franc Forstneri\v{c} and \'Alfhei{\dh}ur Edda Sigur{\dh}ard\'ottir for all the discussions during the writing of this paper.
I am also grateful to the anonymous reviewers for their helpful suggestions and comments.


\section{Notation} \label{notation}

Let $\N$ denote the set of natural numbers, $\N_0 = \N \cup \set{0}$ and $\Z$ the set of integers.
For a point $z \in \C$ and $r > 0$ let $B(z,r)$ be the open disc centered at $z$ with radius $r$.
We also denote the distance between a point $z \in \C$ and a subset $A \subseteq \C$ by 
\[\dist(z,A) = \inf_{w \in A} \abs{z-w}.\]

From now on $X$ will denote a connected open Riemann surface. 
We also fix a complete distance function $d_X \colon X \times X \to [0,\infty)$ on $X$ induced by a smooth Riemannian metric on $TX$.
For a subset $A \subseteq X$ and $r>0$ we then define  
\[A(r) = \set{p \in X \mid \text{$d_X(p,q) < r$ for some $q \in A$}}.\] 
Given an open set $U \subseteq X$ and a map $\gamma \colon U \to X$ we define  
\[\norm{\gamma - \id}_U = \sup_{p \in U} d_X(\gamma(p),p).\] 
For a compact set $K \subseteq X$ and two continuous functions $f,g \colon K \to \C$ we similarly define 
\[\norm{f-g}_K = \sup_{p \in K}\abs{f(p)-g(p)}.\] 

For a set $A \subseteq X$ let $\mathring{A}$, $\partial A$ and $\overline{A}$ denote the topological interior, boundary and closure of $A$ respectively.
If $V \subseteq U \subseteq X$ are open sets, we say that $V$ is \emph{compactly contained} in $U$, if $\overline{V} \subseteq U$ and $\overline{V}$ is compact. 
A \emph{compact exhaustion} of $X$ is a sequence of compact sets $\set{K_n}_{n \in \N_0}$ such that $K_n \subseteq \mathring{K}_{n+1}$ for all $n \in \N_0$ and $\bigcup_{n \in \N_0} K_n = X$. 

For a set $A \subseteq X$ let $\mathcal{C}(A)$ be the class of continuous complex-valued functions on $A$.
For an open set $U \subseteq X$ let $\mathcal{O}(U)$ be the class of holomorphic functions on $U$.
For a compact set $K \subseteq X$ let $\mathcal{O}(K)$ denote the class of holomorphic functions which are defined on some open neighbourhood of $K$.
Finally, for a closed set $E \subseteq X$ let $\mathcal{A}(E)$ denote the class of continuous complex-valued functions on $E$ which are holomorphic on the interior of $E$,
that is $\mathcal{A}(E) = \mathcal{C}(E) \cap \mathcal{O}(\mathring{E})$. In particular if the closed set $E$ has empty interior, we have $\mathcal{A}(E) = \mathcal{C}(E)$.

Let $U \subseteq X$ be an open set, $f \in \mathcal{O}(U)$ a holomorphic function and let $d_pf \colon T_pU \to \C$ be its differential at $p \in U$. 
A point $p \in U$ is called a \emph{critical point} of the function $f$, if $d_pf = 0$.
On the other hand, if $df_p \neq 0$ holds for all $p \in U$, we say that the holomorphic function $f$ is \emph{non-critical}.
In the case when $X = \C$ the function $f$ being non-critical just means that the derivative of the functions $f$ is nowhere vanishing. 
For a closed set $E$, we say that a function in $\mathcal{A}(E)$ is non-critical, if it is non-critical where it is holomorphic.

\section{Approximation theory on Riemann surfaces} \label{ApproximationOnRS}

In this section we will recall some basic notions of approximation theory on Riemann surfaces.

We begin by introducing some terminology.
We say that a closed set $E \subseteq X$ is a set of \emph{uniform approximation}, if for every $\varepsilon > 0$ and $f \in \mathcal{A}(E)$ 
there exists a global holomorphic function $F \in \mathcal{O}(X)$ such that $\abs{F(p)-f(p)} < \varepsilon$ for every $p \in E$.

Let $E \subseteq X$ be a closed set. 
A relatively compact connected component of $X \setminus E$ is called a \emph{hole} of the closed set $E$.
We will denote the union of all the holes of the closed set $E$ by $h(E)$. 

Let $K \subseteq X$ be a compact set. 
The \emph{holomorphically convex hull} or \emph{$\mathcal{O}(X)$-hull} of $K$, denoted by $\widehat{K}$, is defined as 
\[\widehat{K} = \set{p \in X \mid \abs{f(p)} \le \sup_{q \in K}\abs{f(q)}, f \in \mathcal{O}(X)}.\]
The set $\widehat{K}$ is again compact.
If $K = \widehat{K}$, we say that the compact set $K$ is \emph{holomorphically convex}, \emph{$\mathcal{O}(X)$-convex} or \emph{Runge}.
Note that $\widehat{\emptyset} = \emptyset$ because $\sup \emptyset = - \infty$.
Note also that due to the Runge approximation theorem on Riemann surfaces, one has $\widehat{K} = K \cup h(K)$.
Thus a compact set $K$ is $\mathcal{O}(X)$-convex if and only if it has no holes, which is a purely topological condition.
This equivalence only holds on Riemann surfaces, since in higher dimensions the $\mathcal{O}(X)$-hull of a compact set also depends on the complex structure.

By $X^* = X \cup \set{\infty}$ we will denote the one-point compactification of $X$.
It is easy to see that a closed subset $E \subseteq X$ has no holes if and only if $X^* \setminus E$ is connected in $X^*$.

It is sometimes more convenient to work with this characterization.
For example, if $E, E' \subseteq X$ are two closed subsets without holes, then their intersection $E \cap E'$ also has no holes. 
Indeed, by the above characterization both $X^*\setminus E$ and $X^*\setminus F$ are connected.
Since $\infty \in (X^*\setminus E) \cap (X^*\setminus E')$, the set $(X^*\setminus E) \cup (X^*\setminus E') = X^* \setminus (E \cap E')$ is also connected and thus $E \cap E'$ has no holes.

Let $E \subseteq X$ be a closed set. We say that $E$ has the \emph{bounded exhaustion hull property} or \emph{BEH property},
if for every compact set $K$, the set $h(E \cup K)$ is relatively compact. 
This definition is slightly different from the standard one given in the literature, for example in \cite[168]{MinimalSurfacesFromAnalyticPOV},
but it agrees on Riemann surfaces. This follows from the fact that a compact set is Runge if and only if it has no holes.

For a non-example take $X = \C$ and let $E = \overline{\set{x + f(x)i \mid x \in (0,\infty)}}$ be the closure of the graph of the function $f(x) = \frac{1}{x}\sin\left(\frac{1}{x}\right)$ on $(0,\infty)$ as a subset of $\C$. 
It is easy to see that $E$ does not have the BEH property.

Note that the BEH property admits the following characterization using the one-point compactification of $X$.

\begin{proposition}\label{locallyconnected}
    Let $E \subseteq X$ be a closed subset. 
    Then $E$ has the BEH property if and only if $X^*\setminus E$ is locally connected at $\infty$.
\end{proposition}

\begin{proof}
    First assume that $E$ has the BEH property. 
    Let $X^* \setminus K$ be some open neighborhood of $\infty$ in $X^*$, where $K \subseteq X$ is some compact set.
    By the BEH property, $\overline{h(E \cup K)}$ is a compact set. 
    Now set $L = K \cup \overline{h(E \cup K)}$ to get an open neighborhood $X^*\setminus L \subseteq X^*\setminus K$ of $\infty$ in $X^*$.
    We claim $X^* \setminus (E \cup L)$ is connected, which is equivalent to the set $E \cup L$ not having holes.
    So suppose $U$ is a connected component of $X \setminus (E \cup L)$.
    Notice that $U \subseteq X \setminus (E \cup L) \subseteq  X \setminus (E \cup K)$, so $U$ is contained is some connected component $V$ of $X \setminus (E \cup K)$.
    The set $V$ is not relatively compact, since if it were, we would have $U \subseteq V \subseteq h(E \cup K) \subseteq L$, which is a contradiction.
    In particular this implies $V \cap h(E \cup K) = \emptyset$, and since $V \subseteq X \setminus (E \cup K)$, we have $V \cap L = \emptyset$.
    Then
    \[V \subseteq X \setminus (E \cup K) = (X \setminus (E \cup L)) \cup (L \setminus (E \cup K))\]
    implies $U \subseteq V \subseteq X \setminus (E \cup L)$. 
    Since $U$ is a maximal connected component of $X \setminus (E \cup L)$ and $V$ is connected, we have $U = V$ and hence the set $U$ is not relatively compact.

    Now assume that $X^*\setminus E$ is locally connected at $\infty$ and let $K \subseteq X$ be any compact set. 
    Then $X^* \setminus (E \cup K)$ is an open neighborhood of $\infty$ in $X^*\setminus E$. 
    Since by our assumption $X^*\setminus E$ is locally connected at $\infty$, there exists a compact set $L \subseteq X$ such that $K \subseteq L$ and $X^* \setminus (E \cup L)$ is connected, or equivalently that the set $E \cup L$ has no holes. 
    Now let $U$ be a hole of $K \cup E$ and note $U \setminus L \subseteq X \setminus (E \cup L)$.
    We claim that any connected component of $U \setminus L$ is a connected components of $X \setminus (E \cup L)$ as well.
    To see this, let $V$ be a connected component of $U \setminus L$ and let $\widetilde{V}$ be the connected component of $X \setminus (E \cup L)$ that contains $V$.
    Since $\widetilde{V}$ is a connected subset of $X \setminus (E \cup K)$, $U \cap \widetilde{V} \supseteq V \neq \emptyset$ and $U$ is a connected component of $X \setminus (E \cup K)$, we must have $\widetilde{V} \subseteq U$.  
    Furthermore, $\widetilde{V} \subseteq X \setminus (E \cup L)$ implies $\widetilde{V} \subseteq U \setminus (E \cup L) = U \setminus L$, so by the connectedness of $\widetilde{V}$, it follows $V = \widetilde{V}$, proving the claim.
    Since $U$ is relatively compact by assumption, any connected component of $U \setminus L$ would also be relatively compact,
    but since $E \cup L$ has no holes, this implies $U \setminus L = \emptyset$.
    It follows that $h(K \cup E) \subseteq L$ and since $L$ is compact, $h(K \cup E)$ is relatively compact.
\end{proof}

P.~M.~Gauthier and W.~Hengartner showed in \cite{Gauthier_Hengartner} that $X^* \setminus E$ being connected and locally connected are necessary conditions for $E$ to be a set of uniform approximation, see also \cite{NeedConnectedAndLocConnected}.
If the set $E$ is compact, it is clear that it has the BEH property. 
Thus Mergelyan's theorem, see \cite[Theorem 3]{LegacyOfWeirstrass}, says that compact sets of uniform approximation are precisely Runge sets.
For non-compact closed set the problem of uniform approximation is more subtle. 
In the case when $X = \C$, the set $\mathbb{CP}^1 \setminus E$ being connected and locally connected is also sufficient and is the content of Arakelyan approximation, see \cite[Theorem 10]{LegacyOfWeirstrass}. 
For general Riemann surfaces $X^* \setminus E$ being connected and locally connected is not a sufficient condition as was shown in \cite{Gauthier_Hengartner}, see also \cite{NeedConnectedAndLocConnected}.

We now turn our attention to the problem of Carleman approximation on Riemann surfaces.
We say that a closed set $E \subseteq X$ is a set of \emph{Carleman approximation}, if for every continuous positive-valued function $\varepsilon \colon E \to (0,\infty)$ and $f \in \mathcal{A}(E)$
there exists a global holomorphic function $F \in \mathcal{O}(X)$ such that $\abs{F(p)-f(p)} < \varepsilon(p)$ holds for every $p \in E$.

P.~Gauthier showed in \cite{Gauthier} that if a closed set $E \subseteq X$ is a set of Carleman approximation, it enjoys the following property: 
For any compact set $K \subseteq X$ there exists a compact set $K \subseteq Q \subseteq X$ such that no connected component of $\mathring{E}$ intersects both $K$ and $X \setminus Q$.
If a closed set $E$ has the above property, we say it \emph{satisfies condition $\mathcal{G}$}. Note that if the closed set $E$ has empty interior, it trivially satisfies condition $\mathcal{G}$.
Sets of Carleman approximation on Riemann surfaces were later characterized by A.~Boivin in \cite{Boivin1986}.
He showed that a closed set $E \subseteq X$ is a set of Carleman approximation if and only if $X^*\setminus E$ is connected, locally connected and $E$ satisfies condition $\mathcal{G}$.

More recently Carleman type approximation results were developed for the study of minimal surfaces, see \cite[Chapter 3.8]{MinimalSurfacesFromAnalyticPOV} and \cite{BrettCastro}.
In particular, I.~Castro-Infantes and B.~Chenoweth proved a Carleman type approximation result for directed minimal immersion on so called admissible sets.
We say that a closed set $E \subseteq X$ is \emph{admissible} if it is of the form $E = H \cup S$ where $H$ is the union of a locally finite pairwise disjoint family of smoothly bounded compact domains 
and $S$ is the union of a locally finite pairwise disjoint collection of smooth Jordan arcs, such that if a given arc in $S$ intersects the boundary of a compact domain in $H$, 
it intersects it only at its endpoints and does so transversely. 

Lastly we consider the problem of non-critical holomorphic approximation. 
The goal is to approximate a non-critical holomorphic function defined on a subset of the open Riemann surface $X$ by a global non-critical holomorphic function.  
One of the main results in this area is the following Runge approximation theorem for non-critical holomorphic functions due to F.~Forstnerič.
\begin{theorem} \label{Franci_X}
    Let $K \subseteq X$ be a compact Runge set and let $f \in \mathcal{O}(K)$ be a non-critical holomorphic function.
    Then for every $\varepsilon > 0$ there exists a global non-critical holomorphic function $F \in \mathcal{O}(X)$ such that $\norm{F-f}_K < \varepsilon$.
\end{theorem}  
This result is one of the main technical tools used in the proof of our main theorem. 
In particular it handles all the geometrical nuances of being on a Riemann surface.
For a proof of the theorem, see \cite[Theorem 2.1]{franc} or \cite[Theorem 9.13.10]{StainManifoldsAndHolomorphicMappings}. 
 
\section{Semi-admissible sets and the class \texorpdfstring{$\widetilde{\mathcal{A}}$}{~A}} \label{terminology}

We now define the class of sets which we will consider for our approximation result.

\begin{definition}\label{semi-admissible}
    We say that a closed set $E \subseteq X$ is \emph{semi-admissible}, 
    if there exists a locally finite pairwise disjoint family of compact sets $\set{H_\lambda}_{\lambda \in \Lambda}$
    and a closed set $S$ with empty interior, such that $E = S \cup H$, where $H = \bigcup_{\lambda \in \Lambda} H_\lambda$.
\end{definition}
Note that the sets $H$ and $S$ in the above definition are not uniquely determined, but we can always assume $H = \overline{\mathring{E}}$ and $S = \overline{E \setminus H}$. 
In particular the compact sets $H_\lambda$ have non-empty interior.
Let us also remark that when a semi-admissible set $E \subseteq X$ has no holes, the same holds for the compact sets $H_\lambda$.

Note that we do not assume any conditions on the compact sets $H_\lambda$ making up the set $H$.
They need not be connected and their connected components are allowed to accumulate. 
We only require that the compact sets $H_\lambda$ are nicely separated from each other.

The notion of a semi-admissible sets is similar to that of an admissible set, but without the smoothness assumptions.
In particular, any admissible set is semi-admissible. The converse doesn't hold. 
Consider the set $E_1 \subseteq \C$ given by 
\[E_1 = \set{0} \cup \bigcup_{n \in \N} \overline{B\left(\frac{1}{n}, \frac{1}{2n(n+1)}\right)},\]
see figure \ref{Primer1}.
It is semi-admissible, since one can see it as one compact set making up $H$, but it is not admissible.
In particular any compact set is semi-admissible.

For a non-example consider the set $E_2 \subseteq \C$ given by 
\[E_2 = \bigcup_{n \in \Z} \overline{B\left(n, \frac{1}{2}\right)},\]
see figure \ref{Primer2}. 
It is easy to see that the set $E_2$ is not semi-admissible, 
but it is a set of Carleman approximation, the later of which follows from Boivin's characterization.
This demonstrates the main feature of semi-admissible sets, namely that there is some room between the compact sets making it up.
This is crucially used in the proof of the main theorem and is the main difficulty in achieving non-critical Carleman approximation on general sets of Carleman approximation. 
For a further discussion of this topic see Section \ref{Uniform}.  

\begin{figure}[H]
    \centering
    \begin{minipage}{0.49\textwidth}
        \centering
        \includegraphics[width=0.8\textwidth]{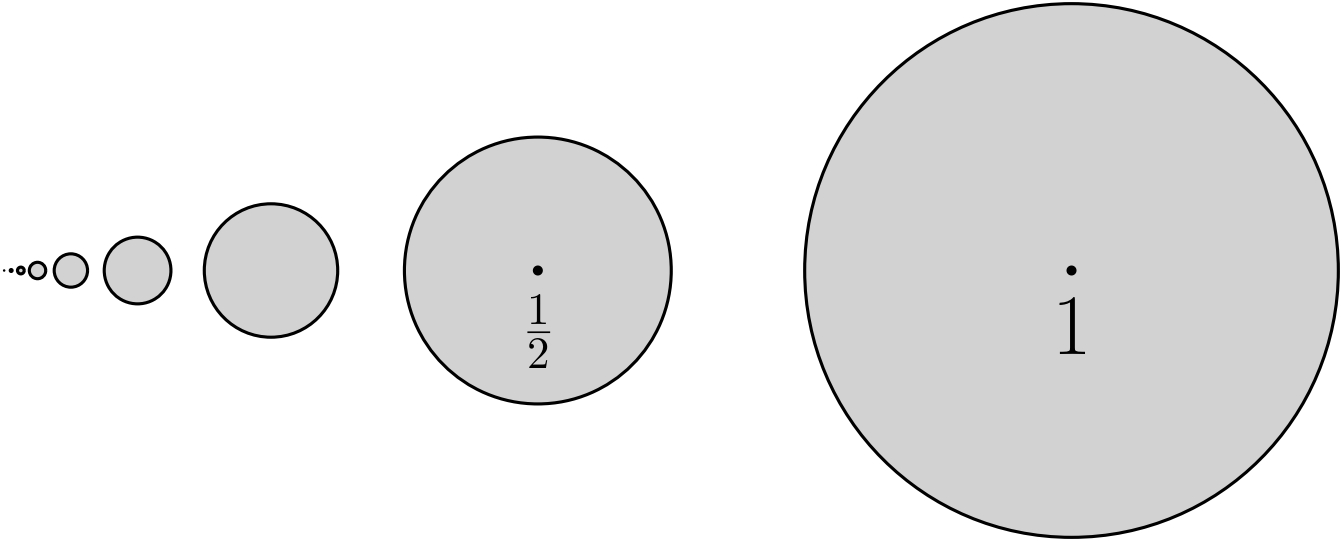}
        \captionsetup{width=0.95\linewidth}
        \caption{The set $E_1$ is semi-admissible, but not admissible.}
        \label{Primer1} 
    \end{minipage}
    \hfill
    \begin{minipage}{0.49\textwidth}
        \centering
        \includegraphics[width=\textwidth]{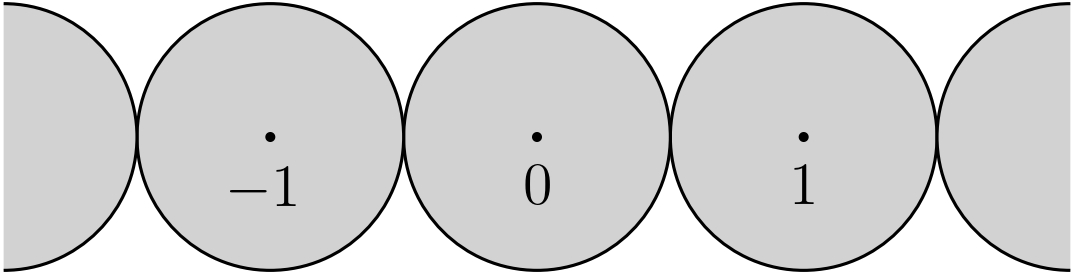}
        \captionsetup{width=\linewidth}
        \caption{The set $E_2$ is a set of Carleman approximation, but it is not semi-admissible.}
        \label{Primer2} 
    \end{minipage}
\end{figure}

It turns out that semi-admissible sets always satisfy condition $\mathcal{G}$. 
This is the content of the next proposition.

\begin{proposition}
    A semi-admissible set $E \subseteq X$ satisfies condition $\mathcal{G}$. 
\end{proposition}

\begin{proof}
    Suppose $E = H \cup S$ where $H = \bigcup_{\lambda \in \Lambda} H_\lambda$ is a locally finite union of pairwise disjoint compact sets.
    Let $K \subseteq X$ be any compact set. For each $p \in K$ there exists a neighborhood $U_p$ of $p$ that intersects at most finitely many $H_\lambda$.
    Since $K$ is compact, the open cover $\set{U_p}_{p \in K}$ has a finite subcover and only finitely many of the sets $H_\lambda$ can intersect the union of that subcover.
    It follows that at most finitely many sets $\set{H_{\lambda_l}}_{l=1}^m$ intersect the compact $K$. 
    Now define the compact set $Q = K \cup \bigcup_{l = 1}^m H_{\lambda_l}$ and note $K \subseteq Q$. 
    Suppose $V \subseteq \mathring{E}$ is a connected component of the interior of $E$.
    Then $V \subseteq \mathring{H}_\lambda$ for some $\lambda \in \Lambda$. 
    If $V$ intersect $K$, then by construction $H_\lambda = H_{\lambda_l}$ for some $l = 1, \ldots, m$ and thus $V \subseteq H_{\lambda_l} \subseteq Q$, hence $V \cap \left(X \setminus Q \right) = \emptyset$ as required.  
\end{proof}

We now define the class of functions that we will consider in our approximation result.

\begin{definition} \label{Atilde}
    Let $E \subseteq X$ be a semi-admissible set of the form $E = H \cup S$. 
    We define $\widetilde{\mathcal{A}}(E)$ be the class of continuous functions on $E$ which are holomorphic on some open neighborhood of the set $H$.
\end{definition}
Note that $\widetilde{\mathcal{A}}(E) \subseteq \mathcal{A}(E)$.
We say that a function $f \in \widetilde{\mathcal{A}}(E)$ is \emph{non-critical}, if it is non-critical on a neighborhood of $H$.
The advantage of working with a function $f \in \widetilde{\mathcal{A}}(E)$ is that the control over the differential of $f$ extends past the boundary of the compact sets $H_\lambda$ making up the semi-admissible set $E$.
This is also the reason why we can allow the connected components of the compact sets $H_\lambda$ to accumulate, as seen in figure \ref{Primer1}. 
In particular, if $E$ is a set of uniform approximation and $f \in \widetilde{\mathcal{A}}(E)$ is non-critical, we can, by approximating well enough, find a global holomorphic function which will be close to $f$ on $E$ and non-critical in a neighborhood of finitely many compact sets $H_\lambda$.
In general, if $f$ is just of the class $\mathcal{A}(E)$, we do not have control over its differential near the boundary. 
The best we can do in this case, by just approximating, is to obtain a global holomorphic function which is close to $f$ on $E$ and non-critical on a compact set contained in $\mathring{H}$. 

\section{Proof of main theorem} \label{MainProof}

We give an outline of the proof of the main theorem. Let $E$ be a semi-admissible set such that $X^* \setminus E$ is connected and locally connected, and choose $f \in \widetilde{\mathcal{A}}(E)$.
First we will show that there exists a compact exhaustion $\set{K_n}_{n \in \N}$ of $X$ by $\mathcal{O}(X$)-convex sets 
such that for $n \in \N$ any compact set $H_\lambda$ making up the semi-admissible set $E$ either lies in the interior of $K_n$ or in the complement of $K_n$.

Denote $E_0 = E$ and $E_n = E \cup K_n \cup h(E \cup K_n)$ for $n \in \N$, note that these sets are semi-admissible. 
We will inductively construct a sequence of functions $\set{f_n}_{n \in \N_0}$, equivalently a sequence of pairs $(f_{n-1}, f_n)$ for $n \in \N$, such that $f_n \in \widetilde{\mathcal{A}}(E_n)$ is non-critical and for $n \in \N$ we have $\norm{f_n - f_{n-1}}_{E_{n-1} \cap K_{n+1}} < \tilde{\varepsilon}_{n-1}$ and $f_n = f_{n-1}$ on $E_n \setminus K_{n+1}$.
By choosing the values $\tilde{\varepsilon}_n$ sufficiently small, 
the sequence $\set{f_n}_{n \in \N_0}$ will converge uniformly to a global non-critical holomorphic function with the desired approximation property.     

To begin the proof, we set $f_0 = f$.
Next suppose we have already constructed the function $f_{n-1} \in \widetilde{\mathcal{A}}(E_{n-1})$. 
First we will use Mergelyan approximation to obtain a holomorphic function $h \in \mathcal{O}(X)$ which approximate $f_{n-1}$ on the Runge set $E_{n-1} \cap K_{n+1}$.
By approximating on a slightly bigger compact set well enough, we can make sure that $h$ is non-critical on a neighborhood of the compact sets $H_\lambda$ making up the semi-admissible set $E_{n-1}$, which are contained inside $K_{n+1}$.
It still might happen that $h$ has critical points on the parts of $E_{n-1}$ with empty interior. 
To fix this, we will slightly perturb the function $h$ to move its critical points off of $E_{n-1}$ and so obtain a non-critical holomorphic function $\tilde{h}$ defined on some open neighborhood of $E_{n-1} \cap K_{n+1}$.
Using Theorem \ref{Franci_X} we will then obtain a global non-critical holomorphic function $g \in \mathcal{O}(X)$ which approximates $f_{n-1}$ on $E_{n-1} \cap K_{n+1}$.
Finally, to obtain $f_n$, we will interpolate between $g$ and $f_{n-1}$ using a smooth bump function. 
We will make sure this interpolation happens along a part of $E_n$ with empty interior so as to guarantee that $f_n$ is of the class $\widetilde{\mathcal{A}}(E_n)$.

First we prove the existence of the compact exhaustion of $X$ that was described above. This is the content of the next two lemmas.

\begin{lemma} \label{obstajaV}
    Let $\set{H_\lambda}_{\lambda \in \Lambda}$ be a locally finite pairwise disjoint family of compact sets in $X$.
    Given a compact set $K \subseteq X$ and an open neighborhood $U$ of $K$, there exists an open neighborhood $V$ of $K$ compactly contained in $U$ 
    such that $H_\lambda \cap \overline{V} \neq \emptyset$ implies $H_\lambda \cap K \neq \emptyset$.
\end{lemma}

\begin{proof}
    Since the family of compact sets $\set{H_\lambda}_{\lambda \in \Lambda}$ is locally finite, 
    there exists for each $p \in K$ some open neighborhood $U_p$ of $p$ compactly contained in $U$ which intersects at most finitely many compact sets $H_\lambda$. 
    Since the set $K$ is compact, the open cover $\set{U_p}_{p \in K}$ has a finite subcover and its union intersects at most finitely many compact sets $H_\lambda$.
    Define the open set $\widetilde{U}$ by removing from this union those compact sets $H_\lambda$ which do not meet $K$.
    A compactly contained neighborhood $V$ of $K$ inside $\widetilde{U}$ will have the desired properties. 
\end{proof}

\begin{lemma}\label{izcrpanje}
    Let $E = H \cup S$ be a semi-admissible set such that $X^* \setminus E$ is connected and locally connected.
    Then there exists an exhaustion $\set{K_n}_{n \in \N_0}$ of $X$ by $\mathcal{O}(X)$-convex compact sets, 
    such that $K_0 = \emptyset$ and for each $n \in \N$ and $\lambda \in \Lambda$ we have:
    \begin{enumerate}[(i)]
        \item $h(E \cup K_n) = \emptyset$,
        \item If $H_\lambda \cap K_n \neq \emptyset$, then $H_\lambda \subseteq K_n$.
    \end{enumerate}
\end{lemma}

\begin{proof}
    Let $\set{H_\lambda}_{\lambda \in \Lambda}$ be the locally finite pairwise disjoint family of compact sets that makes up $H$ 
    and let $\set{\Delta_\nu}_{\nu \in \N_0}$ be an exhaustion of $X$ by $\mathcal{O}(X)$-convex compact sets. 
    We will construct the exhaustion $\set{K_n}_{n \in \N_0}$ of $X$ and an associated sequence of positive integers $\set{\nu_n}_{n \in \N_0}$ inductively as follows.
    For $n = 0$ let $\nu_0 = 0$ and $K_0 = \emptyset$.
    For $n \ge 1$ let $\nu_n > \nu_{n-1}$ be large enough, such that $\Delta_{\nu_n}$ contains $K_{n-1}$ in its interior.
    Let $L_n = \Delta_{\nu_n} \cup \overline{h(E \cup \Delta_{\nu_n})}$. Note that $L_n$ is compact, since $E$ has the BEH property, and that $h(E \cup L_n) = \emptyset$.
    Since the family $\set{H_\lambda}_{\lambda \in \Lambda}$ is locally finite, 
    only finitely many compact sets $H_\lambda$ intersect $L_n$. Let $K_n$ be the union of $L_n$ and these compact sets.
    Since $H_\lambda \subseteq E$, we have $E \cup K_n = E \cup L_n$.
    
    The sets $K_n$ are clearly compact and form an exhaustion of $X$. 
    Property (i) holds, since $h(E \cup K_n) = h(E \cup L_n) = \emptyset$, and property (ii) holds by construction.
    Thus we need only check that a given set $K_n$ is $\mathcal{O}(X)$-convex.
    To this end let $U$ be a connected component of $X \setminus K_n$.
    If $U \subseteq E$, then $U \subseteq H_\lambda$ for some $\lambda \in \Lambda$ and in particular $\partial U \cap H_\lambda \neq \emptyset$.
    Since $\partial U \subseteq K_n$, this implies that $H_\lambda$ must be one of the compact sets making up $K_n$. It follows that $U \subseteq K_n$, which is a contradiction. 
    Thus the set $(X \setminus E) \cap U$ is non-empty and we can choose $V$ to be one of its connected components.
    Observe that 
    \begin{align*}
        V& \subseteq X \setminus (E \cup K_n) \subseteq X \setminus (E \cup \Delta_{\nu_n})\\
        \partial V& \subseteq \partial (X \setminus E) \cup \partial U \subseteq E \cup \partial K_n = E \cup \partial L_n \subseteq E \cup \Delta_{\nu_n},
    \end{align*}
    which together implies that $V$ is a connected component of $X \setminus (E \cup \Delta_{\nu_n})$.
    Since $V \subseteq U$ and $U \cap h(E \cup \Delta_{\nu_n}) = \emptyset$, the set $V$ is not relatively compact, hence the same must hold for the set $U$.
    This proves that $K_n$ has no holes, therefore it is $\mathcal{O}(X)$-convex.
\end{proof}

Next we show that approximating a non-critical holomorphic function on a compact set well enough yields a function which is also non-critical on that set 
and that a suitable limit of non-critical holomorphic functions is again non-critical.
Although these are standard results in approximation theory, we provide proofs for the sake of completion. 

\begin{lemma} \label{exists_r}
    Let $K, H \subseteq X$ be sets such that $K$ is compact and that $K \subseteq \mathring{H}$. 
    Let $\set{(U_j,\phi_j)}_{j=1}^l$ be a finite cover of $K$ with charts. 
    Then there exists some $0 < r \le 1$, such that for every $p \in K$ there is some chart $(U_j, \phi_j)$ such that $p \in U_j$ and
    \[B(\phi_j(p),r) \subseteq \phi_j(H \cap U_j).\] 
\end{lemma}

\begin{proof} 
    For $j = 1, \ldots, l$ define $d_j \colon K \to [0, \infty)$ by 
    \[d_j(p) = \begin{cases} 
    \dist(\phi_j(p), \phi_j(H \cap U_j)^c); \quad p \in U_j \\
    0; \quad p \notin U_j
    \end{cases}\] 
    and note it is a continuous function. If $p \in K \cap U_j$, then $\dist(\phi_j(p), \phi_j(H \cap U_j)^c) > 0$, since $K \subseteq \mathring{H}$.
    Now define $d \colon K \to [0, \infty)$ by 
    \[d(p) = \max_{j=1,\ldots,l} d_j(p).\]
    Note that this function is also continuous and strictly positive, since each $p \in K$ is contained in some neighborhood $U_j$.
    Thus the function $d$ attains a positive minimum $m$ on the compact set $K$. We claim $r = \min \set{m,1}$ is the desired number. 
    Indeed, suppose there exists some $p \in K$ such that $B(\phi_j(p),r) \cap \phi_j(H \cap U_j)^c \neq \emptyset$ for all $j = 1, \ldots, l$ such that $p \in U_j$.
    This implies $d_j(p) < r$ for all $j = 1, \ldots, l$ and hence $d(p) < r \le m$ as well, which is a contradiction. This concludes the proof.
\end{proof}

\begin{lemma}\label{MergelyanSetup}
    Let $K, H \subseteq X$ be compact sets such that $K \subseteq \mathring{H}$ and let $f \in \mathcal{O}(U)$ be a non-critical holomorphic function on some open neighborhood $U$ of $H$.
    Then there exists some $\delta > 0$ such that if a holomorphic function $g \in \mathcal{O}(U)$ satisfies $\norm{f-g}_{H} < \delta$, then $g$ is non-critical on $K$.
\end{lemma}

\begin{proof}
    Let $\set{(U_j,\phi_j)}_{j=1}^l$ be a finite cover of $H$ with charts. 
    Choose an open cover $\set{V_j}_{j=1}^l$ of $H$ such that the set $V_j$ is compactly contained in $U_j$.
    Since $f$ is non-critical, the function $\abs{(f \circ (\phi_j)^{-1})'}$ is strictly positive on the compact sets $H \cap \overline{V_j}$ and thus attains a positive minimum $m^j$.
    Denote $m = \min \set{m^j \mid j=1, \ldots, l}$ and note $m > 0$.
    Let $r > 0$ be given by lemma \ref{exists_r} for the compact sets $K \subseteq H$ and charts $\set{(V_j,\phi_j)}_{j=1}^l$.
    We claim that $\delta = \frac{mr}{4}$ is the desired number. Indeed, assume $g \in \mathcal{O}(U)$ is a holomorphic function such that $\norm{f-g}_{H} < \delta$ and pick $p \in K$.
    Let $(V_j,\phi_j)$ be the chart corresponding to $p$ by lemma \ref{exists_r}. 
    Since $B(\phi_j(p),r) \subseteq \phi_j(H \cap V_j)$, we have $\partial B\left(\phi_j(p),\frac{r}{2}\right) \subseteq \phi_j(H \cap V_j)$.
    Thus we can use the fact that $\abs{(f - g) \circ \phi_j^{-1}} < \delta$ on $\partial B\left(\phi_j(p),\frac{r}{2}\right)$ and Cauchy estimates to obtain
    \begin{align*}
        \abs{((f-g) \circ (\phi_j)^{-1})'(\phi_j(p))} &\le \frac{1}{2\pi} \int_{\partial B\left(\phi_j(p),\frac{r}{2}\right)} \frac{\abs{(f-g) \circ (\phi_j)^{-1}(\zeta)}}{\abs{\zeta - \phi_j(p)}^2} \, \abs{d \zeta} \\
        &\le \frac{1}{2\pi} \int_{\partial B\left(\phi_j(p),\frac{r}{2}\right)} \frac{4\delta}{r^2}\, \abs{d \zeta} =\frac{2\delta}{r} = \frac{m}{2}. 
    \end{align*}
    It now follows
    \begin{align*}
    \abs{(g \circ (\phi_j)^{-1})'(\phi_j(p))} &\ge \abs{(f \circ (\phi_j)^{-1})'(\phi_j(p))} - \abs{((f-g) \circ (\phi_j)^{-1})'(\phi_j(p))}\\
    &\ge m - \frac{m}{2} = \frac{m}{2} > 0,
    \end{align*}
    which shows that $p$ is not a critical point of $g$. It follows that $g$ is non-critical on $K$.
\end{proof}

\begin{lemma} \label{appox_lemma}
    Let $\set{K_n}_{n \in \N_0}$ be a compact exhaustion of the Riemann surface $X$ and let $\set{f_n}_{n \in \N_0}$ be a sequence of functions,
    such that $f_n \in \mathcal{O}(K_n)$ are non-critical. Then there exists a sequence of positive numbers $\set{\delta_n}_{n \in \N_0}$
    such that if $\norm{f_n-f_{n-1}}_{K_{n-1}} < \delta_{n-1}$ for each $n \in \N$, 
    the point-wise limit of the functions $\set{f_n}_{n \in \N_0}$ converges uniformly on compact sets of $X$ to a global non-critical holomorphic function $F = \lim_{n \to \infty} f_n$.
\end{lemma}

\begin{proof} 
    For each $n \in \N$ let $\tilde{\delta}_n$ be given by Lemma \ref{MergelyanSetup} for the compact sets $K_{n-1} \subseteq K_n$ and the holomorphic function $f_n$.
    Let $\set{\delta_n}_{n \in \N_0}$ be a sequence of positive numbers, such that $\sum_{k=n}^\infty \delta_k < \tilde{\delta}_n$ holds for each $n \in \N$.
    Fix some $n \in \N$. Since 
    \[\sum_{k=n}^{\infty} \norm{f_{k+1}-f_k}_{K_{n-1}} \le \sum_{k=n}^{\infty} \norm{f_{k+1}-f_k}_{K_k} \le \sum_{k=n}^{\infty} \delta_k < \tilde{\delta}_n < \infty,\]
    the series $\sum_{k=n}^{\infty} (f_{k+1}-f_k)$ converges uniformly on $K_{n-1}$.
    For $N > n$ we can then write $f_N = f_n + \sum_{k=n}^{N-1} (f_{k+1}-f_k)$ and by taking the limit as $N \to \infty$, we get
    \[F = f_n + \sum_{n=k}^{\infty} (f_{k+1}-f_k).\]
    Since $F$ is a sum of two holomorphic functions, $F$ is holomorphic on $\mathring{K}_{n-1}$.
    The fact that $F$ is also non-critical on $\mathring{K}_{n-1}$ follows directly from 
    \[\norm{F - f_n}_{K_{n-1}} \le \sum_{k=n}^{\infty} \norm{f_{k+1}-f_k}_{K_{n-1}} < \tilde{\delta}_n,\]
    and our choice of $\tilde{\delta}_n$ in Lemma \ref{MergelyanSetup}.
    Since $n \in \N$ was arbitrary and the compact sets $\set{K_n}_{n \in \N_0}$ exhaust $X$, 
    this shows that $F$ really is a global non-critical holomorphic function.
\end{proof}

In the next lemma we construct a biholomorphic map which moves points off of sets with empty interior.
The map will be given as a flow of a holomorphic vector field.

\begin{lemma}\label{bigMove}
    Let $K \subseteq X$ be a compact set, $U$ an open neighborhood of $K$ and $\delta > 0$.
    Suppose there are points $p_1, p_2, \ldots, p_n \in K$ such that for each $j = 1,\ldots, n$ there exists an open neighborhood $W_j$ of $p_j$ such that $W_j \cap K \subseteq \partial K$. 
    Then there exist some open neighborhood $V \subseteq U$ of $K$ and a biholomorphic map $\gamma \colon V \to \gamma(V) \subseteq X$, 
    such that $\norm{\gamma - \id}_V < \delta$ and $p_j \notin \gamma(K)$ for all $j = 1,\ldots, n$.
\end{lemma}

\begin{proof}
    Let $v \colon X \to TX$ be any holomorphic vector field, such that $v(p_j) \neq 0$ holds for all $j= 1, \ldots, n$. 
    Then there exists some $\tau > 0$ and some open neighborhood $\widetilde{V}$ of $K$ such that the flow map
    $\Phi \colon B(0,\tau) \times \widetilde{V} \to X$ of the vector field $v$ is well-defined and holomorphic.
    In fact, for any $t \in B(0,\tau)$ the time-$t$ map on $\widetilde{V}$ given by $p \mapsto \Phi(t, p)$ is biholomorphic onto its image.      
    Let $V$ be an open neighborhood of $K$ that is contained in $U$ and compactly contained in $\widetilde{V}$. 
    Since the map $\Phi \colon B(0,\tau) \times V \to X$ is jointly continuous, the map $t \mapsto \norm{\Phi(t, \cdot) - \id}_V$ is continuous.
    So by choosing $\tau$ small enough, we may assume $\norm{\Phi(t, \cdot) - \id}_V < \delta$ for any $t \in B(0,\tau)$.
    Choosing $\tau$ even smaller if necessary, we can also assume $\Phi(t, p_j) \in W_j$ for all $t \in B(0,\tau)$ and $j=1, \ldots, n$.

    For $j=1, \ldots, n$ we now define a holomorphic map $\phi_j \colon B(0,\tau) \to W_j$ by 
    \[\phi_j(t) = \Phi(t, p_j).\]
    Since 
    \[\frac{d \phi_j}{dt} (0) = \left(\frac{\partial}{\partial t} \Phi(\cdot, p_j)\right)(0) = v(\Phi(0,p_j))= v(p_j) \neq 0,\] 
    the map $\phi_j$ can not be constant. Thus $\phi_j$ is an open map and since $W_j \setminus K$ is dense in $W_j$, the set $\phi_j^{-1}(W_j \setminus K)$ is open and dense in $B(0,\tau)$.
    Since a finite intersection of dense open sets is non-empty, there exists some $t_0 \in \bigcap_{j=1}^n \phi_j^{-1}(W_j \setminus K)$.
    Then the time map $\gamma = \Phi(-t_0, \cdot)$ is the desired map, 
    since $\gamma \colon V \to \gamma(V)$ is a biholomorphic map and for each $j=1, \ldots, n$ we have 
    \[\gamma^{-1}(p_j) = (\Phi(-t_0, \cdot))^{-1}(p_j) = \Phi(t_0, p_j) = \phi_j(t_0) \notin K. \qedhere\]
\end{proof}

We are now ready to prove the main theorem.

\begin{proof}[Proof of main theorem:]
    Since $E$ is semi-admissible, it is of the form $E = H \cup S$ where $H$ is the union of the family of compact sets $\set{H_\lambda}_{\lambda \in \Lambda}$. 
    Let $\set{K_n}_{n \in \N_0}$ be the compact exhaustion of $X$ by $\mathcal{O}(X)$-convex sets given by Lemma \ref{izcrpanje} for the family $\set{H_\lambda}_{\lambda \in \Lambda}$.
    For $n \in \N_0$ we define the closed sets $E_n = E \cup K_n$. 
    Note that the sets $E_n$ are semi-admissible and by property (i) of the exhaustion $\set{K_n}_{n \in \N}$ they have no holes.
    In particular the set $E_{n-1} \cap K_{n+1}$ is Runge for every $n \in \N$.
    Note that we also have $E_n \setminus K_{n+1} = E \setminus K_{n+1}$ for $n \in \N_0$.

    Since $\varepsilon$ is a positive-valued continuous function, it achieves a positive minimum on compact sets. 
    For $n \in \N_0$ we can thus define $\varepsilon_n = \min \set{\varepsilon(p) \mid p \in E \cap K_{n+2}}$. 
    Note that $0 < \varepsilon_n \le \varepsilon_{n-1}$ for $n \in \N$.
    Furthermore let $\set{\delta_n}_{n \in \N_0}$ be given by Lemma \ref{appox_lemma} for the compact exhaustion $\set{K_n}_{n \in \N_0}$.
    For $n \in \N_0$ we then define 
    \[\tilde{\varepsilon}_n = \min\set{2^{-n-1}\varepsilon_n, \delta_n}.\]

    We will inductively construct a sequence of functions $\set{f_n}_{n \in \N_0}$ such that the following conditions hold:
    \begin{samepage}
        \begin{enumerate}[(1)]
            \item $f_n \in \widetilde{\mathcal{A}}(E_n)$ and $f_n$ is non-critical,
            \item $\norm{f_n - f_{n-1}}_{E_{n-1} \cap K_{n+1}} < \tilde{\varepsilon}_{n-1}$,
            \item $f_n = f_{n-1}$ on $E \setminus K_{n+1}$.
        \end{enumerate}
    \end{samepage}

    For $n=0$ take $f_0 = f$. 
    Condition (1) holds since $f \in \widetilde{\mathcal{A}}(E)$ is non-critical by assumption.
    Conditions (2) and (3) in this case are trivially satisfied.
    
    Now suppose we have already constructed $f_{n-1} \in \widetilde{\mathcal{A}}(E_{n-1})$ and let $L = E_{n-1} \cap K_{n+1}$.
    Note that the set $L$ has no holes and that $f_{n-1}$ is holomorphic and non-critical on some open neighborhood of $\overline{\mathring{L}}$.
    Choose compact sets $H_1, H_2 \subseteq X$ such that $\overline{\mathring{L}} \subseteq \mathring{H}_1 \subseteq H_1 \subseteq \mathring{H}_2$, 
    the set $L \cup H_2$ is Runge and that $f_{n-1}$ is holomorphic and non-critical on a neighborhood of $H_2$. 
    Let $\delta_1$ be given by Lemma \ref{MergelyanSetup} applied to $H_1 \subseteq H_2$ and $f_{n-1}$.
    We use Mergelyan's approximation theorem on the Runge set $L \cup H_2$ and the function $f_{n-1}$ to obtain a global holomorphic function $h \in \mathcal{O}(X)$ such that $\norm{h-f_{n-1}}_{L \cup H_2} < \min \set{\delta_1, \frac{\tilde{\varepsilon}_{n-1}}{3}}$.
    In particular we have 
    \[\norm{h-f_{n-1}}_{L} < \frac{\tilde{\varepsilon}_{n-1}}{3}\]
    and by our choice of $\delta_1$ we know $h$ is non-critical on an open neighborhood of $\overline{\mathring{L}}$.  
    
    Next we choose some $\delta_2 > 0$ such that the critical points of $h$ either lie on $L$ or outside $L(\delta_2)$.
    Let $p_1, \ldots, p_m$ be the critical points of $h$ that lie on $L$, there are only finitely many.
    Since $h$ is non-critical on an open neighborhood of $\overline{\mathring{L}}$, there exists for each $p_j$ some open neighborhood $W_j$, such that $W_j \cap L \subseteq \partial L$.
    We use Lemma \ref{bigMove} on the compact set $L$, the open neighborhood $L(\delta_2)$ and the points $p_1, \ldots, p_m$
    to obtain an open neighborhood $V$ of $L$ and a biholomorphic map $\gamma \colon V \to \gamma(V)$ such that $\norm{\gamma - \id}_V < \delta_2$ and $p_j \notin \gamma(E_{n-1} \cap K_{n+1})$ for all $j = 1,\ldots, m$.
    We define the holomorphic function $\tilde{h} \colon V \to \C$ by $\tilde{h} = h \circ \gamma$.
    By choosing $\delta_2$ sufficiently small, the uniform continuity of $h$ implies
    \[\norm{\tilde{h}-h}_{L} < \frac{\tilde{\varepsilon}_{n-1}}{3}.\]
    It can be verified that the function $\tilde{h}$ has no critical points on $L$, and by restricting $V$ if necessary, 
    we can assume the function $\tilde{h}$ is non-critical on $V$ as well. 

    We now use Theorem \ref{Franci_X} on the Runge set $L$ and the function $\tilde{h}$ to obtain a global non-critical holomorphic function $g \in \mathcal{O}(X)$ such that 
    \[\norm{g-\tilde{h}}_{L} < \frac{\tilde{\varepsilon}_{n-1}}{3}.\]
    By applying the triangle inequality we then obtain.
    \[\norm{g - f_{n-1}}_{L} \le \norm{g - \tilde{h}}_{L} + \norm{\tilde{h} - h}_{L} + \norm{h-f_{n-1}}_{L} < \tilde{\varepsilon}_{n-1}.\] 
   
    Let $U_1$ be an open neighborhood of $K_n$ obtained by applying Lemma \ref{obstajaV} to the compact set $K_n$, the open neighborhood $\mathring{K}_{n+1}$ and the family $\set{H_\lambda}_{\lambda \in \Lambda}$.
    Let $U_2$ be an open neighborhood of $K_n$ which is compactly contained in $U_1$.
    By the choice of $U_1$ and property (ii) of the exhaustion $\set{K_n}_{n \in \N_0}$
    any compact set $H_\lambda$ is contained either in $U_2$ or in $X \setminus \overline{U_1}$. 
    Let $\chi \colon X \to [0,1]$ be a smooth bump function, such that $\supp \chi \subseteq U_1$ and $\chi = 1$ on $\overline{U_2}$.
    We now define the function $f_n$ as
    \[f_n = \chi g + (1-\chi) f_{n-1}.\]
    The function $f_n$ is clearly continuous on $E_n$ and on some open neighborhood of $K_n$ or $H_\lambda$ it will agree with either $f_{n-1}$ inside $X \setminus \overline{U_1}$ or $g$ inside $U_2$.
    This shows that $f_n$ satisfies property (1).
    To verify property (2) we estimate 
    \[\norm{f_n-f_{n-1}}_{L} = \norm{\chi(g-f_{n-1})}_{L} \le \norm{g-f_{n-1}}_{L} < \tilde{\varepsilon}_{n-1},\]
    and to verify property (3), simply notice that $\chi = 0$ on $E_n \setminus K_{n+1}$. 

    We define the function $F$ as the pointwise limit of the sequence $\set{f_n}_{n \in \N_0}$.
    Since in particular $f_n \in \mathcal{O}(K_n)$ are non-critical and 
    \[\norm{f_n - f_{n-1}}_{K_{n-1}} \le \norm{f_n - f_{n-1}}_{L} < \tilde{\varepsilon}_{n-1} < \delta_{n-1},\]
    Lemma \ref{appox_lemma} guarantees that $F$ is a global non-critical holomorphic function. 
    It remains to verify the approximation property. To this end pick $p \in E$ and note that either $p \in K_1$ or $p \in K_{n+1}\setminus K_n$ for some $n \in \N$.
    If $p \in K_1$, we have 
    \[\abs{F(p)-f(p)} \le \sum_{k=0}^\infty \abs{f_{k+1}(p)- f_k(p)} < \sum_{k=0}^\infty \tilde{\varepsilon}_k \le \sum_{k=0}^\infty 2^{-k-2}\varepsilon_k \le \varepsilon_0 \sum_{k=0}^\infty 2^{-k-1} = \varepsilon_0 \le \varepsilon(p).\]
    If $p \in K_{n+1} \setminus K_n$, we have by property (3) that $f_{n-1}(p) = \ldots = f_0(p)$. We can then estimate 
    \begin{align*}
    \abs{F(p)-f(p)} &\le \sum_{k=0}^\infty \abs{f_{k+1}(p)- f_k(p)} = \sum_{k=n-1}^\infty \abs{f_{k+1}(p)- f_k(p)} \\
    &< \sum_{k=n-1}^\infty \tilde{\varepsilon}_k \le \sum_{k=n-1}^\infty 2^{-k-2}\varepsilon_k \le \varepsilon_{n-1} \sum_{k=n-1}^\infty 2^{-k-1} \le \varepsilon_{n-1} \le \varepsilon(p).
    \end{align*}
    Thus $\abs{F(p)-f(p)} < \varepsilon(p)$ for each $p \in E$, which concludes the proof.
\end{proof}

\section{Non-critical approximation on general sets of Carleman approximation} \label{Uniform}

Lastly, we discuss the problem of non-critical approximation on general sets of Carleman approximation and present a different approach 
that yields uniform approximation on some sets of Carleman approximation that are not semi-admissible.

The main reason why we considered semi-admissible sets for non-critical Carleman approximation is that
they have some room between components with non-empty interior. 
Thus, we can interpolate using smooth bump functions and not need to worry about holomorphicity, 
since where we interpolate a function of the class $\widetilde{\mathcal{A}}$ need only be continuous. 

This is not the case for general sets of Carleman approximation. Consider again the set 
\[E_2 = \bigcup_{n \in \Z} \overline{B\left(n, \frac{1}{2}\right)} \subseteq \C,\]
see figure \ref{Primer2}.
Trying to replicate the proof of the main theorem in this setting, 
we would end up interpolating by a smooth bump function on a part of the set $E_2$ with non-empty interior where a function of the class $\widetilde{\mathcal{A}}(E_2)$ must be holomorphic.
This requires us to modify our function by a solution of a $\bar{\partial}$-equation, but by doing so, we lose control over the critical points of the function.
One can consider different solutions to $\bar{\partial}$-equations, for example, ones with Hörmander estimates, see \cite[Section 4.2]{Hormander}, but in this case one loses the approximation properties of the function along the set $E_2$. 
This is the problem one needs to overcome to achieve non-critical approximation on more general sets.

Using a different approach, we can achieve non-critical uniform approximation on certain sets of Carleman approximation in simply connected open Riemann surfaces. 
The idea is to approximate the derivative of our initial function, and then integrate the approximating function to obtain our global non-critical approximation.
The same approach can be used to give an alternate proof of Theorem \ref{Franci_X} in the case when $X = \C$, see \cite[Theorem 3.1]{franc}.

The next Proposition shows how this idea can be applied to prove that non-critical uniform approximation is possible on the set $E_2$

\begin{proposition}
    Let $E_2 \subseteq \C$ be the closed set defined above, $f \in \mathcal{O}(E_2)$ a non-critical holomorphic function and $\varepsilon > 0$.
    Then there exists an entire non-critical function $F \in \mathcal{O}(\C)$ such that $\abs{F(z)-f(z)} < \varepsilon$ holds for every $z \in E_2$.
\end{proposition}

\begin{proof}
For every $z \in E_2$, there exists some $n \in \Z$ such that $z \in \overline{B\left(n, \frac{1}{2}\right)}$.
Let $\gamma_z^1$ be the line segment between $0$ and $n$ and $\gamma_z^2$ the line segment between $n$ and $z$.
Then $\gamma_z = \gamma_z^1 \cup \gamma_z^2$ is a piecewise linear path from the origin to $z$ contained inside $E_2$.

Suppose the function $f$ is holomorphic and non-critical on some neighborhood $U$ of $E_2$.
By restricting $U$ if necessary, we may assume it is simply connected. 
Then the function $f'$ is a nowhere vanishing holomorphic function on a simply connected domain $U$, so it is of the form $f' = e^g$ for some $g \in \mathcal{O}(U)$.
Let $\delta > 0$ and $C > 0$ be such that for $\abs{z} < \delta$ we have $\abs{e^z-1} \le C \abs{z}$. 
Since $E_2$ is a set of Carleman approximation, we can take the continuous complex-valued function $g \colon E_2 \to \C$ and the continuous positive-valued function 
\[\tilde{\varepsilon}(z) = \min\set{\delta, \frac{\varepsilon e^{-\re(z)^2}}{\sqrt{\pi}C\abs{e^{g(z)}}}},\]
and find an entire function $G \in \mathcal{O}(\C)$ such that $\abs{G(z)-g(z)} < \tilde{\varepsilon}(z)$ for all $z \in E_2$.

We define the function $F \colon \C \to \C$ by
\[F(z) = f(0) + \int_0^z e^{G(\zeta)} \, d\zeta.\]
Note that $F$ is a well-defined entire holomorphic function and its value is independent of the choice of path from $0$ to $z$ along which we integrate. Thus, without loss of generality, we can pick the path to be $\gamma_z$.
For $z \in \overline{B\left(n, \frac{1}{2}\right)}$, we now calculate
\begin{align*}
    \abs{F(z) - f(z)} &= \abs{F(z) - f(0) + f(0) - f(z)} = \abs{\int_{\gamma_z} e^{G(\zeta)} \, d\zeta - \int_{\gamma_z} f'(\zeta) \, d\zeta} \\
    &= \abs{\int_{\gamma_z} e^{G(\zeta)} \, d\zeta - \int_{\gamma_z} e^{g(\zeta)} \, d\zeta} \le \int_{\gamma_z} \abs{e^{G(\zeta)} - e^{g(\zeta)}} \, \abs{d\zeta} \\
    &= \int_{\gamma_z} \abs{e^{g(\zeta)}}\abs{e^{G(\zeta)-g(\zeta)} - 1} \, \abs{d\zeta} \le C \int_{\gamma_z} \abs{e^{g(\zeta)}}\abs{G(\zeta)-g(\zeta)} \, \abs{d\zeta} \\
    &\le C\int_{\gamma_z} \abs{e^{g(\zeta)}}\tilde{\varepsilon}(\zeta) \, \abs{d\zeta} \le \frac{\varepsilon}{\sqrt{\pi}} \int_{\gamma_z} e^{-\re(\zeta)^2} \, \abs{d\zeta} \\
    &= \frac{\varepsilon}{\sqrt{\pi}} \left(\int_{\gamma_z^1} e^{-\re(\zeta)^2} \, \abs{d\zeta} + \int_{\gamma_z^2} e^{-\re(\zeta)^2} \, \abs{d\zeta}\right).
\end{align*}
A quick computations shows
\[\int_{\gamma_z^1} e^{-\re(\zeta)^2} \, \abs{d\zeta} = \int_0^n e^{-t^2} \, dt \le \frac{\sqrt{\pi}}{2}\]
and 
\[\int_{\gamma_z^2} e^{-\re(\zeta)^2} \, \abs{d\zeta} = \int_0^1 e^{-(n + t(\re(z)-n))^2} \, \abs{z-n}dt \le \int_0^1 1 \cdot \frac{1}{2} \, dt = \frac{1}{2} \le \frac{\sqrt{\pi}}{2},\]
which give us the desired estimate 
\[\abs{F(z) - f(z)} \le \frac{\varepsilon}{\sqrt{\pi}} \left(\frac{\sqrt{\pi}}{2} + \frac{\sqrt{\pi}}{2}\right) \le \varepsilon.\]
Finally notice that $F' = e^{G} \neq 0$, so $F$ is indeed non-critical. 
\end{proof}

The key ingredient of the proof was that we were able to construct a positive-valued function $\tilde{\varepsilon}$, which could correct for the distance traveled along the paths on which we were integrating. 
Thus, this approach extends to any simply connected open Riemann surface and set of Carleman approximation provided one knows how to construct a suitable function $\tilde{\varepsilon}$. 
We also remark that, in general, using this approach, we cannot do better than uniform approximation.

We do not know if uniform approximation by non-critical functions is possible for more general sets, for example sets on which Arakelyan approximation is possible.
The problem is open even for simple sets, a uniform horizontal strip for example.

\nocite{*} 
\printbibliography

\end{document}